\documentclass[12pt]{amsart}
\usepackage{amssymb,amsmath,amsthm,mathtools,amsfonts,times,hyperref,mathrsfs,multirow}
\usepackage{enumerate}
\usepackage{graphicx}
\usepackage{array}
\usepackage{ytableau}
\usepackage{pstricks,pst-node,graphicx}
\usepackage{tikz,varwidth}
\usepackage{setspace}
\usepackage{float}
\usetikzlibrary{calc}
\usepackage{tikz}
\usepackage{extarrows}

\newtheorem{theorem}{Theorem}[section]
\newtheorem{lemma}[theorem]{Lemma}

\theoremstyle{definition}
\newtheorem{definition}[theorem]{Definition}
\newtheorem{example}[theorem]{Example}

\theoremstyle{remark}
\newtheorem{remark}{Remark}
\numberwithin{equation}{section}

\allowdisplaybreaks

\newcommand\numberthis{\addtocounter{equation}{1}\tag{\theequation}}
\newcommand\restr[2]{{
  \left.\kern-\nulldelimiterspace 
  #1 
  \littletaller 
  \right|_{#2} 
  }}
\newcommand{\littletaller}{\mathchoice{\vphantom{\big|}}{}{}{}}  

\begin{document}

\title[]
{A bijective proof of an identity of Berkovich and Uncu}

\author{Aritram Dhar}
\address{Department of Mathematics, University of Florida, Gainesville
FL 32611, USA}
\email{aritramdhar@ufl.edu}
\author{Avi Mukhopadhyay}
\address{Department of Mathematics, University of Florida, Gainesville
FL 32611, USA}
\email{mukhopadhyay.avi@ufl.edu}

\date{\today}



\begin{abstract}
The BG-rank BG($\pi$) of an integer partition $\pi$ is defined as $$\text{BG}(\pi) := i-j$$ where $i$ is the number of odd-indexed odd parts and $j$ is the number of even-indexed odd parts of $\pi$. In a recent work, Fu and Tang ask for a direct combinatorial proof of the following identity of Berkovich and Uncu $$B_{2N+\nu}(k,q)=q^{2k^2-k}\left[\begin{matrix}2N+\nu\\N+k\end{matrix}\right]_{q^2}$$ for any integer $k$ and non-negative integer $N$ where $\nu\in \{0,1\}$, $B_N(k,q)$ is the generating function for partitions into distinct parts less than or equal to $N$ with BG-rank equal to $k$ and $\left[\begin{matrix}a+b\\b\end{matrix}\right]_q$ is a Gaussian binomial coefficient. In this paper, we provide a bijective proof of Berkovich and Uncu's identity along the lines of Vandervelde and Fu and Tang's idea.
\end{abstract}
\maketitle

\section{Introduction}\label{s1}
An integer partition is a non-increasing finite sequence $\pi = (\lambda_1,\lambda_2,\ldots)$ of non-negative integers where $\lambda_i$'s are called the parts of $\pi$. We denote the number of parts of $\pi$ by $\#(\pi)$ and the largest part of $\pi$ by $l(\pi)$. The size of $\pi$ is the sum of the parts of $\pi$ and is denoted by $|\pi|$. We say that $\pi$ is a partition of $n$ if $|\pi| = n$. $\lambda_{2i-1}$ (resp. $\lambda_{2i}$) are called odd-indexed (resp. even-indexed) parts of $\pi$.\\\par In \cite{Ber-Gar06} and \cite{Ber-Gar08}, Berkovich and Garvan defined the BG-rank of a partition $\pi$, denoted by $BG(\pi)$, as 
\begin{align*}
    BG(\pi) := i-j,
\end{align*}
where $i$ is the number of odd-indexed odd parts and $j$ is the number of even-indexed odd parts. The BG-rank of a partition $\pi$ can also be represented by $2$-residue Ferrers diagram of $\pi$. The $2$-residue Ferrers diagram of a partition $\pi$ is represented by writing the ordinary Ferrers diagram with boxes instead of dots and filling the boxes using alternate $0$’s and $1$’s starting from $0$ on odd-indexed parts and $1$ on even-indexed parts. In \cite{Ber-Gar08}, Berkovich and Garvan showed that
\begin{align*}
    BG(\pi) = r_0-r_1,
\end{align*}
where $r_0$ (resp. $r_1$) is the number of $0$'s (resp. $1$'s) in the $2$-residue Ferrers diagram of $\pi$. For example, Figure \ref{fig1} below depicts the $2$-residue Ferrers diagram for the partition $\pi = (10,7,4,2)$ and so, $BG(\pi) = r_0-r_1 = 11-12 = -1$.
\begin{figure}[H]
\centering
  \ytableaushort
      {{0}{1}{0}{1}{0}{1}{0}{1}{0}{1},{1}{0}{1}{0}{1}{0}{1},{0}{1}{0}{1},{1}{0}}
      * {10,7,4,2}
  \caption{$2$-residue Ferrers diagram of the partition $\pi = (10,7,4,2)$}
  \label{fig1}
\end{figure}
Let $L,m,n$ be non-negative integers. We now recall some notations from the theory of $q$-series that can be found in \cite{And98}.
\begin{align*}
    (a)_L = (a;q)_L &:= \prod_{k=0}^{L-1}(1-aq^k),\\
    (a)_{\infty} = (a;q)_{\infty} &:= \lim_{L\rightarrow \infty}(a)_L\,\,\text{where}\,\,\lvert q\rvert<1.
\end{align*}
We define the $q$-binomial (Gaussian) coefficient as
\begin{align*}
    \left[\begin{matrix}m\\n\end{matrix}\right]_q := \Bigg\{\begin{array}{lr}
        \dfrac{(q)_m}{(q)_n(q)_{m-n}}\quad\text{for } m\ge n\ge 0,\\
        0\qquad\qquad\quad\text{otherwise}.\end{array}
\end{align*}
\begin{remark}\label{rmk1}
For $m, n\ge 0$, $\left[\begin{matrix}m+n\\n\end{matrix}\right]_q$ is the generating function for partitions into at most $n$ parts each of size at most $m$ (see \cite[Chapter $3$]{And98}). Also, note that $$\lim_{m\longrightarrow\infty}\left[\begin{matrix}m+n\\n\end{matrix}\right]_q = \dfrac{1}{(q)_n}.$$
\end{remark}
\par We call a partition into distinct parts a \textit{strict} partition. If $B_N(k,q)$ denotes the generating function for the number of strict partitions into parts less than or equal to $N$ with BG-rank equal to $k$, then Berkovich and Uncu \cite{Ber-Unc16} showed that for any non-negative integer $N$ and any integer $k$,
\begin{align}\label{eq1}
    B_{2N+\nu}(k,q) = q^{2k^2-k}\left[\begin{matrix}2N+\nu\\N+k\end{matrix}\right]_{q^2},
\end{align}
where $\nu\in\{0,1\}$. Letting $N\rightarrow \infty$ in (\ref{eq1}), we have
\begin{align}\label{eq2}
    \sum_{n=0}^{\infty}p_k^d(n)q^n = \dfrac{q^{2k^2-k}}{(q^2;q^2)_{\infty}},
\end{align}
where $p_k^d(n)$ denotes the number of strict partitions of $n$ with BG-rank equal to $k$. Note that (\ref{eq2}) is exactly Conjecture 1 in \cite{Van10} where Vandervelde defined a partition statistic called characteristic, denoted by $\chi(\pi)$, which is related to BG-rank as
\begin{align*}
    \chi(\pi) = -BG(\pi).
\end{align*}
In \cite[Remark $3.9$]{Fu-Tang20}, Fu and Tang mention that (\ref{eq2}) can also be derived from the work of Boulet \cite{Bou06}. Setting $a=d=qz$, $b=c=q/z$, and $z=1$ in Corollary 2 in \cite{Bou06}, we get (\ref{eq2}). In \cite{Van10}, Vandervelde provided a bijective proof of (\ref{eq2}) with $k = 0$. More precisely, \cite[Theorem $1$]{Van10} states
\begin{align}\label{eq3}
    \sum_{n=0}^{\infty}p_0^d(n)q^n = \dfrac{1}{(q^2;q^2)_{\infty}}.
\end{align}
Building upon Vandervelde's bijection, Fu and Tang \cite{Fu-Tang20} provided a bijective proof of (\ref{eq2}) for all integers $k$ using certain unimodal sequences whose alternating sum equals zero. In their paper \cite[Remark $3.9$]{Fu-Tang20}, Fu and Tang ask for a direct combinatorial proof of (\ref{eq1}). The main aim of this paper is to provide such a combinatorial (bijective) proof.\\\par The rest of the paper is organized as follows. In Section \ref{s2}, we present Fu and Tang's bijection. In Section \ref{s3}, we present the proof of Berkovich and Uncu's identity (\ref{eq1}). In Section \ref{s4}, we provide some examples to illustrate the bijective proof of (\ref{eq1}). We conclude with a few remarks in Section \ref{s5} to motivate further investigation.\\

\section{Fu and Tang's bijection}\label{s2}
\subsection{$(a,b)$-sequences}\label{s21}\quad\\\par
First, we will define a certain type of unimodal sequence called an $(a,b)$-sequence introduced by Fu and Tang \cite[Definition $2.1$]{Fu-Tang20}.
\begin{definition}
    For some non-negative integer $a$ and an integer $1\leq b\leq l$, we call a sequence of $l$ positive integers $\{d_1,\ldots,d_l\}$ an $(a,b)$-\textit{sequence of length} $l$ if the following conditions hold:
    \begin{enumerate}
        \item $d_{i} = a+i$ for $1\leq i\leq b$,
        \item $d_{i}$ forms a non-increasing sequence of positive integers for $i\geq b$, and
        \item $\sum\limits_{i=1}^{l} (-1)^{i}d_{i}=0$.\\
    \end{enumerate}
\end{definition} We denote the collection of all such sequences by $\mathcal{S}_{a,b}$ and define $\mathcal{S} := (\bigcup_{a\geq 0, b\geq 1} \mathcal{S}_{a,b})\cup\{\varepsilon\}$ where $\varepsilon$ is the empty sequence. For $\Delta = \{d_1,\ldots,d_l\}$, we denote $l(\Delta) = l$, $|\Delta| = \sum\limits_{i=1}^ld_i$, and $|\Delta|_{\text{alt}} = \sum\limits_{i=1}^l(-1)^id_i$. If $\Delta\in\mathcal{S}_{a,b}$, we denote $a(\Delta) = a$ and $b(\Delta) = b$.\\
\begin{example}
    $\{5,6,7,8,3,3,2,2,2,1,1\}$ is a $(4,4)$-sequence of length $11$.\\    
\end{example}

\subsection{The Bijection $\phi_a$}\label{s22}\quad\\\par
According to Chu \cite{Chu93}, a $k$-Durfee rectangle for the Young diagram of a partition is an $i\times (i+k)$ rectangle (having $i$ rows and $i+k$ columns) which is obtained by choosing the largest possible $i$ such that the $i\times (i+k)$ rectangle is contained in the Young diagram for a fixed integer $k$. It is to be noted that Fu and Tang \cite{Fu-Tang20} mention that this notion of Durfee rectangle is different from the generalization by Andrews in \cite{And71}.\\\par For integers $a \geq 0$ and $b \geq 1$, we consider a map $\phi_a : \mathcal{S}_{a,b}\to\mathcal{P}_{a,b}$ where $\mathcal{P}_{a,b}$ is the set of all integer partitions $\lambda = (\lambda_{1},\lambda_{2},\ldots)$ whose $a$-Durfee rectangle has size $\lceil\frac{b}{2}\rceil\times(\lceil\frac{b}{2}\rceil + a)$ and $\lambda_{\frac{b}{2}} > a + b/2 $ if $b$ is even or $\lambda_{\frac{b+1}{2}} = a + (b+1)/2$ if $b$ is odd.\\\par Now, we will define the map $\phi_a$. Consider a sequence $\Delta = \{d_1,d_2,\ldots,d_l\}\in \mathcal{S}_{a,b}$. The aim is to use the sequence $\Delta\in\mathcal{S}_{a,b}$ to \textit{double cover} the block diagram configuration shown in Figure \ref{fig2} below. 
\begin{figure}[H]
\centering
  \includegraphics[width=.4\linewidth]{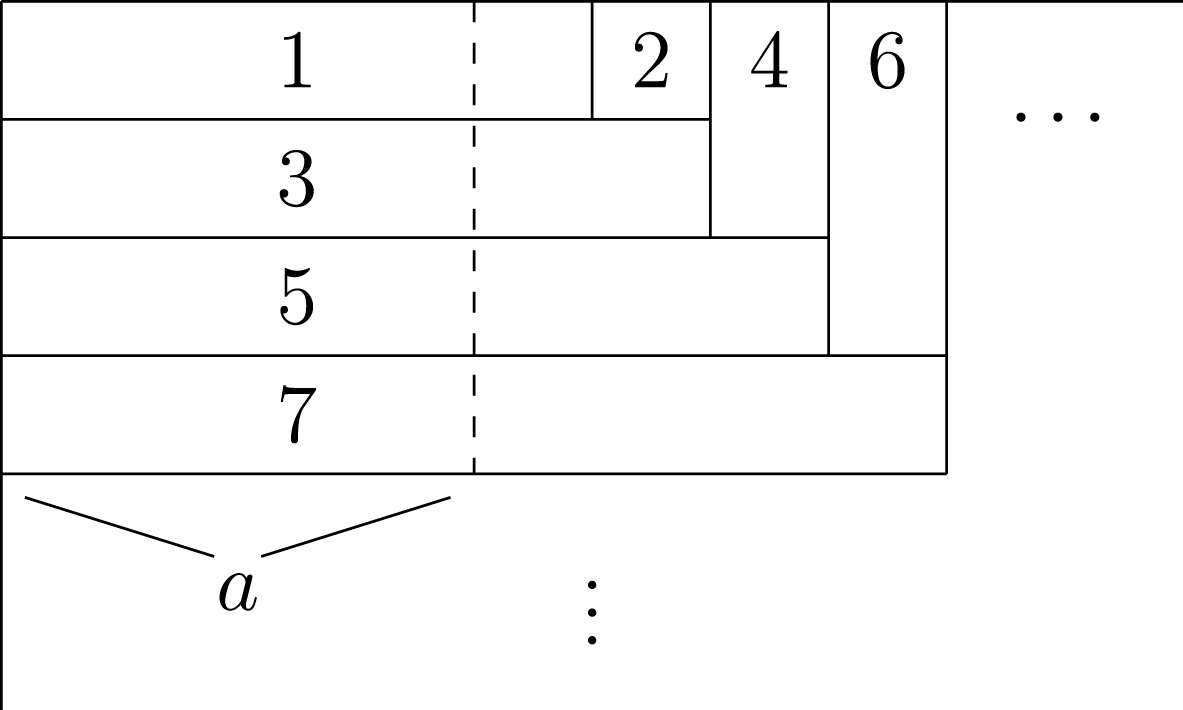}
  \caption{Block diagram configuration for $\phi_a$ with labeled blocks}
  \label{fig2}
\end{figure}\par Following Fu and Tang \cite[Fig. $2$]{Fu-Tang20}, in the block diagram configuration (see Figure \ref{fig2} above), we call the $i$th labeled block $B_i$. $B_i$ has size $1\times \left(a + \dfrac{i+1}{2}\right)$ (resp. $\dfrac{i}{2}\times 1$) if $i$ is odd (resp. even). We denote the area of $B_i$, i.e., the number of cells, by $n_i$. So, $n_1 = a+1,\,n_2 = 1,\,n_3 = a+2,\,n_4 = 2$, and so on. We obtain $\phi_{a}(\Delta)$ by performing the following operations:\\
\begin{enumerate}
    \item Fill up $B_1$ in the block diagram Figure \ref{fig2} with $d_1 = a+1$ cells which is equivalent to labeling the $a+1$ cells in $B_1$ with `$1$'.\\
    \item Use $d_i$ cells first to \textit{double cover} the already existing cells in $B_{i-1}$ for $2\leq i\leq l$ and then use the remaining cells to fill $B_i$. This is equivalent to using $d_i$ cells to re-label the already existing $n_{i-1}$ cells in $B_{i-1}$  by `$2$' first for $2\leq i\leq l$ and then labeling the remaining $d_i-n_{i-1}$ cells by `$1$' to fill $B_i$.\\ 
    \item Filling of $B_i$'s (labeling by `$1$' and re-labeling by `$2$') are done from left to right if $i$ is odd and from top to bottom if $i$ is even.\\
    \item After having used up all the $d_i$'s where $1\leq i\leq l$, the \textit{doubly covered} cells (cells which are labeled by `$2$') form the Young diagram of a partition (say) $\lambda = \phi_a(\Delta)$.\\
\end{enumerate}
\par The notion of double covering of the cells in the block diagram configuration is equivalent to coloring the cells by yellow (or labeling them by a `$1$') and then re-coloring the cells by green (or re-labeling them again by a `$2$') so that in the end, all the cells are colored in green (or labeled by `$2$'). This is exactly the reason why the base in the $q$-binomial coefficient in (\ref{eq1}) is $q^2$ instead of just $q$ as we are counting the cells twice. We then call a block diagram \textit{doubly covered} when all the cells are colored green. The doubly covered block diagram will then be the Young diagram of a partition in $\mathcal{P}_{a,b}$. From now onwards, $b_i$ will denote the number of doubly covered cells (colored green) labeled by $\mathcal{B}_i$.\\
\begin{example}
    Suppose $a = 3$, $b = 2$, and $\Delta = \{4,5,2,1\}\in \mathcal{S}_{3,2}$. Then following steps ($1$) to ($4$) above, we have $\lambda = \phi_3(\Delta) = (5,1)\in\mathcal{P}_{3,2}$ and $|\lambda| = |\Delta|/2 = 6$. For an illustration, see Figure \ref{fig3} below where the intermediate steps are denoted by arrows from left to right. All \textit{singly covered} (equivalent to being labeled by `$1$' or counted once) cells are colored yellow and all \textit{doubly covered} (equivalent to being labeled by `$2$' or counted twice) cells are colored green. All the cells labeled $\mathcal{B}_i$ form a sub-region of the $i$th block $B_i$ and $b_i$ is the number of doubly covered cells (colored green) labeled $\mathcal{B}_i$ for $i\in\{1,2,3\}$. Here, $b_1 = 4$, $b_2 = 1$, and $b_3 = 1$ form the partition $\lambda = (5,1)$. 
\begin{figure}[H]
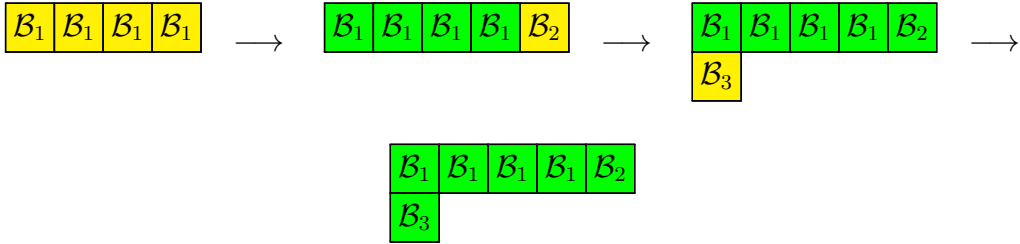

    \centering
      \ytableaushort
      {{\mathcal{B}_1}{\mathcal{B}_1}{\mathcal{B}_1}{\mathcal{B}_1}}
      * {4}
      * [*(yellow)]{4}
      \quad
      $\longrightarrow$
      \quad
      \ytableaushort
      {{\mathcal{B}_1}{\mathcal{B}_1}{\mathcal{B}_1}{\mathcal{B}_1}{\mathcal{B}_2}}
      * {5}
      * [*(green)]{4}* [*(yellow)]{1+4}
      \quad
      $\longrightarrow$
      \quad
      \ytableaushort
      {{\mathcal{B}_1}{\mathcal{B}_1}{\mathcal{B}_1}{\mathcal{B}_1}{\mathcal{B}_2},{\mathcal{B}_3}}
      * {5,1}
      * [*(green)]{5}* [*(yellow)]{0,1}
      \quad
      $\longrightarrow$
      \quad\\\quad\\\quad\\
      \ytableaushort
      {{\mathcal{B}_1}{\mathcal{B}_1}{\mathcal{B}_1}{\mathcal{B}_1}{\mathcal{B}_2},{\mathcal{B}_3}}
      * {5,1}
      * [*(green)]{5,1}
      \caption{Applying $\phi_3$ on $\Delta = \{4,5,2,1\}$ to get the partition $\lambda = (5,1)\in\mathcal{P}_{3,2}$}
      \label{fig3}
\end{figure}
\end{example}
\begin{theorem}(\cite[Theorem $2.5$]{Fu-Tang20})\label{ftbij}
    For a fixed $a\geq 0$ and any $b\geq 1$, the map $\phi_a$ defined above is a bijection from $\mathcal{S}_{a,b}$ to $\mathcal{P}_{a,b}$, such that $|\Delta| = 2|\phi_a(\Delta)|$, for any $\Delta\in\mathcal{S}_{a,b}$.\\
\end{theorem}
\begin{remark}\label{rmk2}
     Since  $b_1 = a + 1 = d_1$ and $b_{i-1} + b_i = d_i$ for $2\le i\le I+1$ where $I$ is the index of the last present block in the block diagram configuration and $b_{I+1} = 0$, we have $\sum\limits_{i=1}^{I}d_i = 2\sum\limits_{i=1}^{I}b_i$ which implies $|\Delta| = 2|\phi_a(\Delta)|$ as in Theorem \ref{ftbij}. This justifies the fact that after using up all the $d_i$'s, none of the cells in the image partition $\lambda = \phi_a(\Delta)$ are \textit{singly covered}, i.e., none of the cells are colored yellow.\\
\end{remark}

\subsection{Application to strict partitions}\label{s23}\quad\\\par
First, we consider the map $\iota:\mathcal{D}\to\mathcal{T}\times\mathcal{S}$ where $\mathcal{D}$ is the set of all strict partitions and $\mathcal{T}$ is the set of all triangular numbers, i.e., $\mathcal{T} := \left\{\dfrac{n(n+1)}{2} : n\in\mathbb{Z}\right\}$. Fu and Tang \cite[Lemma $3.1$]{Fu-Tang20} proved that $\iota$ is in fact an injection.\\\par For any strict partition $\lambda = (\lambda_1,\lambda_2,\ldots,\lambda_r)\in\mathcal{D}$, we consider the shifted Young diagram of $\lambda$, which is the same as the Young diagram of $\lambda$ except that each row after the first row is indented by one box from the row above it. For instance, see the Young diagram in Figure \ref{fig4} whose cells are colored orange.
\begin{figure}[H]
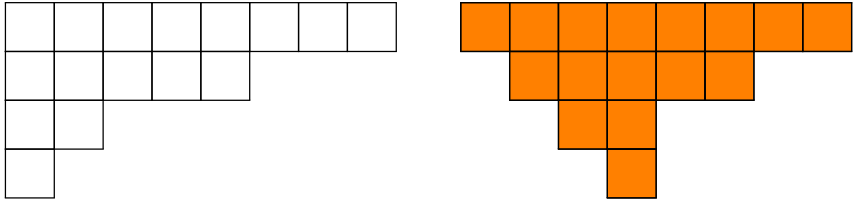

\centering
  \ytableaushort
  {\none,\none,\none,\none}
  * {8,5,2,1}\quad\quad
  \ydiagram{8,1+5,2+2,3+1}
  * [*(orange)]{8,1+5,2+2,3+1}
  \caption{Young diagram and shifted Young diagram representing the partition $\lambda = (8,5,2,1)$}
  \label{fig4}
\end{figure}
Now, construct the sequence of column lengths (read from left to right) of its shifted Young diagram. These column lengths form a unimodal sequence $c(\lambda) = \{c_1,c_2,\ldots,c_{\lambda_1}\}$. For example, for the shifted Young diagram shown in Figure \ref{fig4} above, $c(\lambda) = \{1,2,3,4,2,2,1,1\}$.\\\par Fu and Tang \cite[Lemma $3.1$]{Fu-Tang20} proved that there exists a unique integer $0\le m\le r$ (where $r$ is the number of parts of $\lambda$, i.e., $r = \#(\lambda)$) such that $\sum\limits_{i=1}^m(-1)^ic_i = \sum\limits_{i=1}^{\lambda_1}(-1)^ic_i$, i.e.,  $|\Delta|_\text{alt} = 0$ where $\Delta := \{c_{m+1},c_{m+2},\ldots,c_{\lambda_1}\}$ and so, $\Delta\in\mathcal{S}$.\\\par Now, define $\iota(\lambda) = (t,\Delta)$ where $t = 1+2+\ldots+m = \binom{m+1}{2}$. Clearly, $\iota(\lambda)\in\mathcal{T}\times \mathcal{S}$ and so, $|\lambda| = \lambda_1+\lambda_2+\ldots+\lambda_r = c_1+c_2+\ldots+c_{\lambda_1} = t+|\Delta|$. Fu and Tang \cite[Lemma $3.1$]{Fu-Tang20} proved that $(t,\Delta)\in\iota(\mathcal{D})$ if and only if any one of the following conditions hold
\begin{enumerate}
    \item $a(\Delta) = m$, or
    \item $a(\Delta)\le m-1$ and $b(\Delta) = 1$, or
    \item $\Delta = \varepsilon$.
\end{enumerate} $\iota$ is a bijection simply because the pre-image of any $(t,\Delta)\in\mathcal{T}\times\mathcal{S}$ satisfying either (1) or (2) or (3) mentioned above can be constructed uniquely by appending columns of length $1,2,\ldots,m$ to the left of the columns of length given by the elements of $\Delta$ and obtaining a shifted Young diagram.\\

\section{Bijective proof of Berkovich and Uncu's identity (\ref{eq1})}\label{s3}
We now present the statement of the main result which we prove in this section.\\
\begin{theorem}\label{combbu}
Let $\nu\in\{0,1\}$, $N$ be a non-negative integer, and $k$ be any integer. Then, for any positive integer $n$, the number of strict partitions $\pi_d$ of $n$ with BG-rank equal to $k$ and $l(\pi_d)\le 2N+\nu$ is equal to the number of partitions $\pi$ of $\frac{n-2k^2+k}{2}$ where $l(\pi)\le N+\nu-k$ and $\#(\pi)\le N+k$.\\    
\end{theorem}
Note that Theorem \ref{combbu} together with the partition theoretic interpretation of $q$-binomial coefficient in Remark \ref{rmk1} implies Berkovich and Uncu's identity (\ref{eq1}). We will now provide a bijective proof of Theorem \ref{combbu}.\\
\begin{proof}
Let the set of all strict partitions $\pi_d$ of $n$ having BG-rank equal to $k$ and $l(\pi_d)\le 2N+\nu$ be denoted by $\mathcal{D}_{n,k}^{N,\nu}$, the set of all partitions $\pi$ of $n$ with $l(\pi)\le L$ and $\#(\pi)\le m$ be denoted by $\mathcal{P}_{n,L,m}$, and $T_i = i(i+1)/2$ be the $i$th triangular number for any integer $i$. Clearly, $\mathcal{D}_{n,k}^{N,\nu}\subset \mathcal{D}$ and $T_i\in \mathcal{T}$.\\\par  Consider any $\pi_d\in\mathcal{D}_{n,k}^{N,\nu}$. First, construct the shifted Young diagram of $\pi_d = (\lambda_1,\ldots,\lambda_r)$ $(\lambda_1\le 2N+\nu)$ and then form the unimodal sequence $c(\pi_d) = \{c_1,\ldots,c_{\lambda_1}\}$ where $c_i$ is the length of the $i$th column of the shifted Young diagram of $\pi_d$. There exists $0\le a\le r$ such that $\sum\limits_{i=1}^a(-1)^ic_i = \sum\limits_{i=1}^{\lambda_1}(-1)^ic_i$ and so, $\Delta := \{c_{a+1},c_{a+2},\ldots,c_{\lambda_1}\}\in\mathcal{S}_{a,b}\subset\mathcal{S}$ for some integer $b\ge 1$.\\\par 
\begin{lemma}\label{relka}
For the $\Delta\in\mathcal{S}_{a,b}$ obtained from the shifted Young diagram of $\pi_d\in\mathcal{D}_{n,k}^{N,\nu}$,
    $$a = a(\Delta) = \Bigg\{\begin{array}{lr}
        -2k\quad\quad\text{if}\,\,k\le 0,\\
        2k-1\quad\text{if}\,\,k>0.\end{array}$$
\end{lemma}
\begin{proof}
    For any $\pi_d$ having BG-rank $k$, we have\\
    \begin{align*}
    k &= -\left|\{1,2,\ldots,a\}\right|_{\text{alt}}
    \\
    &= 1-2+3-\ldots+(-1)^{a+1}a
    \\
    &= (-1)^{a+1}\left\lceil\dfrac{a}{2}\right\rceil,
    \end{align*}\\
    where $\lceil x\rceil$ is the least integer greater than or equal to $x$. Hence, $k = -\dfrac{a}{2}$ if $a$ is even which implies $a = -2k$ if $k\le 0$ and $k = \dfrac{a+1}{2}$ if $a$ is odd which implies $a = 2k-1$ if $k > 0$.
\end{proof}\par\quad\par We will now show that for $k\le 0$, $\mathcal{D}_{n,k}^{N,\nu}$ is in bijection with $\mathcal{P}_{\frac{n-2k^2+k}{2},N+\nu-k,N+k}$ and for $k > 0$, $\mathcal{D}_{n,k}^{N,\nu}$ is in bijection with $\mathcal{P}_{\frac{n-2k^2+k}{2},N+k,N+\nu-k}$. The set $\mathcal{P}_{n,L,m}$ is in bijection with $\mathcal{P}_{n,m,L}$, where the bijection is the conjugation map which interchanges the rows and columns of a partition about the main diagonal in the Young's diagram representation of the partition. Consequently, the two sets $\mathcal{P}_{\frac{n-2k^2+k}{2},N+k,N+\nu-k}$ and $\mathcal{P}_{\frac{n-2k^2+k}{2},N+\nu-k,N+k}$ are equinumerous.\\\par\begin{itemize}
    \item \fbox{Case I: $k\le 0$}\\\par One can now easily verify that for Fu and Tang's map $\iota$, $\restr{\iota}{\mathcal{D}_{n,k}^{N,\nu}} : \mathcal{D}_{n,k}^{N,\nu}\longrightarrow\{T_{-2k}\}\times\mathcal{S}_{-2k,b}$ is a bijection where $\restr{\iota}{\mathcal{D}_{n,k}^{N,\nu}}(\pi_d) = (T_{-2k},\Delta)$ with $T_{-2k} = 2k^2-k\in\mathcal{T}$ and $\Delta\in\mathcal{S}_{-2k,b}$.\\\par Now, recall Fu and Tang's bijection $\phi_a$. Consider the map $\chi_{-} : \{T_{-2k}\}\times\mathcal{S}_{-2k,b}\longrightarrow\{T_{-2k}\}\times\mathcal{P}_{\frac{n-2k^2+k}{2},N+\nu-k,N+k}$ defined as $$\chi_{-}(T_{-2k},\Delta) := (T_{-2k},\restr{\phi_{-2k}}{\mathcal{S}_{-2k,b}}(\Delta)).$$ Therefore, we have $\chi_{-}(T_{-2k},\Delta) = (T_{-2k},\pi)$ where $\pi\in\mathcal{P}_{\frac{n-2k^2+k}{2},N+\nu-k,N+k}$. Thus, $\chi_{-}$ is a bijection.\\\par Next, consider the map $\psi_{-} : \mathcal{D}_{n,k}^{N,\nu}\longrightarrow\{T_{-2k}\}\times\mathcal{P}_{\frac{n-2k^2+k}{2},N+\nu-k,N+k}$ defined as $$\psi_{-} := \chi_{-}\circ\restr{\iota}{\mathcal{D}_{n,k}^{N,\nu}}.$$ So, for any $\pi_d\in\mathcal{D}_{n,k}^{N,\nu}$, we have $$\psi_{-}(\pi_d) := \chi_{-}\left(\restr{\iota}{\mathcal{D}_{n,k}^{N,\nu}}(\pi_d)\right) = \chi_{-}(T_{-2k},\Delta) = (T_{-2k},\pi)$$ where $\pi\in\mathcal{P}_{\frac{n-2k^2+k}{2},N+\nu-k,N+k}.$\\\par Clearly, $\psi_{-}$ is an invertible map since it is the composition of two invertible maps $\restr{\iota}{\mathcal{D}_{n,k}^{N,\nu}}$ and $\chi_{-}$ where $\psi_{-}^{-1}$ is given by $$\psi_{-}^{-1} = \left(\restr{\iota}{\mathcal{D}_{n,k}^{N,\nu}}\right)^{-1}\circ\chi_{-}^{-1}.$$
    \item \fbox{Case II: $k > 0$}\\\par Again, it can be verified that for Fu and Tang's map $\iota : \mathcal{D}\longrightarrow\mathcal{T}\times\mathcal{S}$, $\restr{\iota}{\mathcal{D}_{n,k}^{N,\nu}} : \mathcal{D}_{n,k}^{N,\nu}\longrightarrow\{T_{2k-1}\}\times\mathcal{S}_{2k-1,b}$ is a bijection where $\restr{\iota}{\mathcal{D}_{n,k}^{N,\nu}}(\pi_d) = (T_{2k-1},\Delta)$ with $T_{2k-1} = 2k^2-k\in\mathcal{T}$ and $\Delta\in\mathcal{S}_{2k-1,b}$.\\\par Now, recall Fu and Tang's bijection $\phi_a$. Analogous to $\chi_{-}$, consider the map $\chi_{+} : \{T_{2k-1}\}\times\mathcal{S}_{2k-1,b}\longrightarrow\{T_{2k-1}\}\times\mathcal{P}_{\frac{n-2k^2+k}{2},N+k,N+\nu-k}$ defined as $$\chi_{+}(T_{2k-1},\Delta) := (T_{2k-1},\restr{\phi_{2k-1}}{\mathcal{S}_{2k-1,b}^{N,\nu}}(\Delta)).$$ Therefore, we have $\chi_{+}(T_{2k-1},\Delta) = (T_{2k-1},\pi)$ where $\pi\in\mathcal{P}_{\frac{n-2k^2+k}{2},N+k,N+\nu-k}$. Thus, $\chi_{+}$ is a bijection.\\\par Next, consider the map $\psi_{+} : \mathcal{D}_{n,k}^{N,\nu}\longrightarrow\{T_{2k-1}\}\times\mathcal{P}_{\frac{n-2k^2+k}{2},N+k,N+\nu-k}$ defined as $$\psi_{+} := \chi_{+}\circ\restr{\iota}{\mathcal{D}_{n,k}^{N,\nu}}.$$ So, for any $\pi_d\in\mathcal{D}_{n,k}^{N,\nu}$, we have $$\psi_{+}(\pi_d) := \chi_{+}\left(\restr{\iota}{\mathcal{D}_{n,k}^{N,\nu}}(\pi_d)\right) = \chi_{+}(T_{2k-1},\Delta) = (T_{2k-1},\pi)$$ where $\pi\in\mathcal{P}_{\frac{n-2k^2+k}{2},N+k,N+\nu-k}.$\\\par Now, define the map $$\psi^{\ast}_{+} := \ast\circ\psi_{+} = \ast\circ\chi_{+}\circ\restr{\iota}{\mathcal{D}_{n,k}^{N,\nu}}$$ such that $$\psi^{\ast}_{+}(\pi_d) := \ast(\psi_{+}(\pi_d)) = \ast(T_{2k-1},\pi) = (T_{2k-1},\pi^{\ast})$$ where $\ast$ is the conjugation operation and $\pi^{\ast}\in\mathcal{P}_{\frac{n-2k^2+k}{2},N+\nu-k,N+k}$ is the conjugate partition of the partition $\pi\in\mathcal{P}_{\frac{n-2k^2+k}{2},N+k,N+\nu-k}$.\\\par Observe that one can first apply the conjugation operation $\ast$ on the partition $\pi^{\ast}$ to get the conjugate partition $(\pi^{\ast})^{\ast} = \pi$ (since conjugation is an involution) and then apply the inverse map $\psi_{+}^{-1}$ on $(T_{2k-1},\pi)$ to get the strict partition $\pi_d$.\\
\end{itemize} 
\par For a detailed illustration of how the forward (resp. inverse) map $\psi_{-}$ or $\psi_{+}$ (resp. $\psi_{-}^{-1}$ or $\psi_{+}^{-1}$) works, see the examples listed in Section \ref{s4}.\\\par Thus, it is clear that for any strict partition $\pi_d\in\mathcal{D}_{n,k}^{N,\nu}$ of size $|\pi_d| = n$, the image partition $\pi$ has size $|\pi| = \frac{n-2k^2+k}{2}$ and vice-versa.\\\par Finally, we focus our attention on actually obtaining the bounds on the largest part and the number of parts of $\pi$ explicitly under the action of $\psi_{-}$ or $\psi_{+}$. We also show that the we can retrieve back the bound on the largest part of $\pi_d$ explicitly under the action of $\psi_{-}^{-1}$ or $\psi_{+}^{-1}$ on $\pi$. We first present a lemma which lies at the heart of obtaining the desired bounds.\\
\begin{lemma}\label{maxblock}
    The index of the last block present in the block diagram representation of the Young diagram of $\pi$ is at most $l(\pi_d)-a-1$.\\
\end{lemma}
\begin{proof}
    In the shifted Young diagram of $\pi_d$, the length of the unimodal sequence whose alternating sum is zero is equal to $l(\pi_d)-a$. So, the number of blocks that can be \textit{doubly covered} by the elements of this sequence is at most $l(\pi_d)-a-1$.
\end{proof}\par\quad\par Now, we consider two cases according to the sign of the BG-rank $k$ of $\pi_d\in\mathcal{D}_{n,k}^{N,\nu}$.\\
\begin{itemize}
    \item \fbox{Case I: $k\le 0$}\\\par If $k\le 0$, then from Lemma (\ref{relka}), we have $a = -2k$, i.e, $a$ is even.\\\par Let $I$ be the index of the last present block in the block diagram representation of the Young diagram of $\pi$. From Lemma (\ref{maxblock}), we know that
    \begin{align*}
        I & \le l(\pi_d)-a-1\\
        & \le 2N+\nu-a-1\\
        & = 2N+\nu+2k-1\\
        & = 2(N+k)+\nu-1\\
        & = \Bigg\{\begin{array}{lr}
        2(N+k)-1\,\,\text{if}\,\,\nu = 0,\\
        2(N+k)\quad\quad\text{if}\,\,\nu = 1.\end{array}
    \end{align*} Therefore, $\#(\pi)\le N+k$.\\\par Now, let $E$ be the number of even-indexed blocks present in the block diagram representation of the Young diagram of $\pi$. Then, it is clear that $$E\le \sum\limits_{\substack{i=2 \\ 2\mid i}}^{l(\pi_d)-a-1}1.$$\par Again from the block diagram representation of the Young diagram of $\pi$, we have
    \begin{align*}
        l(\pi) & = a+1+E\\
        & \le a+1+\sum\limits_{\substack{i=2 \\ 2\mid i}}^{l(\pi_d)-a-1}1\\
        & \le a+1+\sum\limits_{\substack{i=2 \\ 2\mid i}}^{2N+\nu-a-1}1\\
        & = -2k+1+\sum\limits_{\substack{i=2 \\ 2\mid i}}^{2N+\nu+2k-1}1\\
        & = \Bigg\{\begin{array}{lr}
        -2k+1+\sum\limits_{\substack{i=2 \\ 2\mid i}}^{2N+2k-1}1\quad\,\text{if}\,\,\nu = 0,\\
        -2k+1+\sum\limits_{\substack{i=2 \\ 2\mid i}}^{2N+2k}1\quad\quad\text{if}\,\,\nu = 1\end{array}\\
        & = \Bigg\{\begin{array}{lr}
        -2k+1+N+k-1\,\,\,\,\text{if}\,\,\nu = 0,\\
        -2k+1+N+k\quad\quad\,\,\text{if}\,\,\nu = 1\end{array}\\
        & = \Bigg\{\begin{array}{lr}
        N-k\quad\quad\,\,\text{if}\,\,\nu = 0,\\
        N-k+1\,\,\,\,\text{if}\,\,\nu = 1.\end{array}
    \end{align*} Hence, $l(\pi)\le N+\nu-k$.\\\par For the reverse direction, since $k\le 0$, we know that $a = -2k$, $l(\pi)\le N+\nu-k$, and $\#(\pi)\le N+k$. Clearly, $l(\pi_d) = -2k+l(\Delta) = -2k+I+1$ where $I$ is the index of the last present block in the block diagram representation of the Young diagram of $\pi$. Now, we consider two sub-cases regarding the parity of $I$:\\
    \begin{itemize}
        \item \fbox{Sub-Case IA: $I$ is odd}\\\par Since $\#(\pi)\le N+k$, \begin{align*}
            I&\le 2(N+k)-1\\
            &= 2N+2k-1\\
            &\le 2N+\nu+2k-1\numberthis\label{eq31}
        \end{align*} where (\ref{eq31}) follows from the fact that $\nu\in\{0,1\}$.\\\par Therefore, from (\ref{eq31}), it follows that $l(\pi_d) = -2k+I+1\le 2N+\nu$.\\
        \item \fbox{Sub-Case IB: $I$ is even}\\\par Since $l(\pi)\le N+\nu-k$, \begin{align*}
            I&\le 2((N+\nu-k)-(a+1))\\
            &= 2N+2\nu-2k-2a-2\\
            &= 2N+2\nu+2k-2\\
            &= 2N+\nu+2k-1+\nu-1\\
            &\le 2N+\nu+2k-1+\nu-1+1-\nu\numberthis\label{eq32}\\
            &= 2N+\nu+2k-1\numberthis\label{eq33}
        \end{align*} where (\ref{eq32}) follows from the fact that $1-\nu\in\{0,1\}$.\\\par Therefore, from (\ref{eq33}), it follows that $l(\pi_d) = -2k+I+1\le 2N+\nu$.\\
    \end{itemize}
    \item \fbox{Case II: $k>0$}\\\par If $k>0$, then from Lemma (\ref{relka}), we have $a = 2k-1$, i.e, $a$ is odd.\\\par Let $I$ be the index of the last present block in the block diagram representation of the Young diagram of $\pi$. From Lemma (\ref{maxblock}), we know that
    \begin{align*}
        I & \le l(\pi_d)-a-1\\
        & \le 2N+\nu-a-1\\
        & = 2N+\nu-2k\\
        & = 2(N+\nu-k)-\nu\\
        & = \Bigg\{\begin{array}{lr}
        2(N+\nu-k)\quad\quad\text{if}\,\,\nu = 0,\\
        2(N+\nu-k)-1\,\,\text{if}\,\,\nu = 1.\end{array}
    \end{align*} Therefore, $\#(\pi)\le N+\nu-k$.\\\par If $E$ is the number of even-indexed blocks present in the block diagram representation of the Young diagram of $\pi$, $$E\le \sum\limits_{\substack{i=2 \\ 2\mid i}}^{l(\pi_d)-a-1}1.$$\par Again, from the block diagram representation of the Young diagram of $\pi$, we have
    \begin{align*}
        l(\pi) & = a+1+E\\
        & \le a+1+\sum\limits_{\substack{i=2 \\ 2\mid i}}^{l(\pi_d)-a-1}1\\
        & \le a+1+\sum\limits_{\substack{i=2 \\ 2\mid i}}^{2N+\nu-a-1}1\\
        & = 2k+\sum\limits_{\substack{i=2 \\ 2\mid i}}^{2N+\nu-2k}1\\
        & = \Bigg\{\begin{array}{lr}
        2k+\sum\limits_{\substack{i=2 \\ 2\mid i}}^{2N-2k}1\quad\quad\text{if}\,\,\nu = 0,\\
        2k+\sum\limits_{\substack{i=2 \\ 2\mid i}}^{2N-2k+1}1\quad\,\text{if}\,\,\nu = 1\end{array}\\
        & = \Bigg\{\begin{array}{lr}
        2k+N-k\,\,\,\,\text{if}\,\,\nu = 0,\\
        2k+N-k\,\,\,\,\text{if}\,\,\nu = 1.\end{array}
    \end{align*} Hence, $l(\pi)\le N+k$.\\\par For the reverse direction, since $k > 0$, we know that $a = 2k-1$, $l(\pi)\le N+k$, and $\#(\pi)\le N+\nu-k$. Clearly, $l(\pi_d) = 2k-1+l(\Delta) = 2k-1+I+1 = 2k+I$ where $I$ is the index of the last present block in the block diagram representation of the Young diagram of $\pi$. Now, we consider two sub-cases regarding the parity of $I$:\\
    \begin{itemize}
        \item \fbox{Sub-Case IIA: $I$ is odd}\\\par Since $\#(\pi)\le N+\nu-k$, \begin{align*}
            I&\le 2(N+\nu-k)-1\\
            &= 2N+2\nu-2k-1\\
            &= 2N+\nu-2k+\nu-1\\
            &\le 2N+\nu-2k+\nu-1+1-\nu\numberthis\label{eq34}\\
            &= 2N+\nu-2k\numberthis\label{eq35}
        \end{align*} where (\ref{eq34}) follows from the fact that $1-\nu\in\{0,1\}$.\\\par Therefore, from (\ref{eq35}), it follows that $l(\pi_d) = 2k+I\le 2N+\nu$.\\
        \item \fbox{Sub-Case IIB: $I$ is even}\\\par Since $l(\pi)\le N+k$, \begin{align*}
            I&\le 2((N+k)-(a+1))\\
            &= 2N+2k-2a-2\\
            &= 2N-2k\\
            &= 2N+\nu-2k-\nu\\
            &\le 2N+\nu-2k-\nu+\nu\numberthis\label{eq36}\\
            &= 2N+\nu-2k\numberthis\label{eq37}
        \end{align*} where (\ref{eq36}) follows from the fact that $\nu\in\{0,1\}$.\\\par Therefore, from (\ref{eq37}), it follows that $l(\pi_d) = 2k+I\le 2N+\nu$.\\
    \end{itemize}
\end{itemize}
Thus, we conclude that in the forward direction, $\#(\pi)\le N+k$, $l(\pi)\le N+\nu-k$ if $k\le 0$ and $\#(\pi)\le N+\nu-k$, $l(\pi)\le N+k$ if $k > 0$ and in the reverse direction, $l(\pi_d)\le 2N+\nu$ irrespective of the sign of $k$. This completes the proof of Theorem \ref{combbu}.
\end{proof}\par\quad\par

\section{Examples illustrating Theorem \ref{combbu}}\label{s4}
In this section, we present four different examples where we show the correspondences $\pi_d\underset{\psi_{+}^{-1}}{\stackrel{\psi_{+}}{\rightleftarrows}} (T_a,\pi)$ and $\pi_d\underset{\psi_{-}^{-1}}{\stackrel{\psi_{-}}{\rightleftarrows}} (T_a,\pi)$. Here, $\pi_d\in\mathcal{D}_{n,k}^{N,\nu}$ is a strict partition with fixed $BG$-rank $k$ and $l(\pi_d)\le 2N+\nu$, $T_a = \frac{a(a+1)}{2}$ is the triangular part where $a = a(\Delta)$ with $\Delta = \{d_1,d_2,\ldots,d_{l(\Delta)}\}\in\mathcal{S}_{a,b}$ obtained from the shifted Young diagram of $\pi_d$, $\pi\in\mathcal{P}_{\frac{n-2k^2+k}{2},N+\nu-k,N+k}$ is a partition where $l(\pi)\le N+\nu-k$, $\#(\pi)\le N+k$ if $k\le 0$, and $\pi\in\mathcal{P}_{\frac{n-2k^2+k}{2},N+k,N+\nu-k}$ is a partition where $l(\pi)\le N+k$, $\#(\pi)\le N+\nu-k$ if $k > 0$.\\\par In examples \ref{eg1}, \ref{eg2}, \ref{eg3}, and \ref{eg4}, all \textit{singly covered} (equivalent to being labeled by `$1$' or counted once) cells are colored yellow and all \textit{doubly covered} (equivalent to being labeled by `$2$' or counted twice) cells are colored green. The cells labeled $\mathcal{B}_i$ form a sub-region of the $i$th block $B_i$ and $b_i$ is the number of doubly covered cells (colored green) labeled $\mathcal{B}_i$ for $i\in\mathbb{N}$. In example \ref{eg1}, we show all the intermediate steps (denoted by arrows from left to right) for the forward map in detail. However, in examples \ref{eg2}, \ref{eg3}, and \ref{eg4}, we just portray the strict partition $\pi_d$ and the image $(T_a,\pi)$ without displaying the intermediate steps.
\begin{example}\label{eg1}\quad\\\\
Let $\pi_d = (9,7,5,4,1)\in\mathcal{D}_{26,2}^{4,1}$ so that $l(\pi_d) = 9\le 2N+\nu = 9$. Since $k = 2 > 0$, by Lemma \ref{relka}, $a = 2k-1 = 3$ which implies $T_3 = 6$ is the triangular part. From the shifted Young diagram of $\pi_d$, $c(\pi_d) = \{1,2,3,4,5,4,4,2,1\}$, and $\Delta = \{4,5,4,4,2,1\}$. $\psi_{+}(\pi_d) = (T_3,\pi)$. So, $b_1 = 4$, $b_2 = 1$, $b_3 = 3$, $b_4 = 1$, and $b_5 = 1$ which implies $\pi = (6,3,1)$. Clearly, $l(\pi) = 6 = N+k$ and $\#(\pi) = 3 = N+\nu-k$. Hence, $\pi\in\mathcal{P}_{10,6,3}$.\\\par Now, for the reverse direction, we are given $T_3 = 6$ and $\pi = (6,3,1)\in\mathcal{P}_{10,6,3}$. So, the solutions to $2k^2-k = 6$ are $k = 2$ and $k = -\frac{3}{2}$. Since $k\in\mathbb{Z}$, $k = 2 > 0$. On solving $N+k = 6$ and $N+\nu-k = 3$, we have $(N,\nu) = (4,1)$. On solving $\frac{n-6}{2} = 10$, we have $|\pi_d| = n = 26$. Now, $a = 2\cdot 2-1 = 3$ since $k = 2 > 0$ which implies $b_1 = a+1 = 4$, $b_2 = 1$, $b_3 = 3$, $b_4 = 1$, and $b_5 = 1$ following the block diagram configuration in Figure \ref{fig1}. Now, we obtain $d_1 = b_1 = 4$, $d_2 = b_1+b_2 = 5$, $d_3 = b_2+b_3 = 4$, $d_4 = b_3+b_4 = 4$, $d_5 = b_4+b_5 = 2$, and $d_6 = b_5+b_6 = 1$ since $b_6 = 0$. Thus, we obtain the sequence $\{4,5,4,4,2,1\}$ which we write column-wise and if we append columns of length $\{1,2,3\}$ to the left of the column of length $4$, we retrieve back the shifted Young diagram of the partition $\pi_d = (9,7,5,4,1)\in\mathcal{D}_{26,2}^{4,1}$.\\\quad\\\quad\\
\ytableausetup{centertableaux}
\ytableaushort
{\none,\none,\none,\none}
* {9,7,5,4,1}
$\longrightarrow$
\quad
\ydiagram{9,1+7,2+5,3+4,4+1}
$\longrightarrow$
\\\quad\\\quad\\
\ydiagram{3,1+2,2+1}
\quad,\quad
\ydiagram{6,5,4,4,1+1}
$\longrightarrow$
\quad
\ydiagram{3,1+2,2+1}
* [*(yellow)]{3,1+2,2+1}
\quad,\quad
\ytableaushort
{{\mathcal{B}_1}{\mathcal{B}_1}{\mathcal{B}_1}{\mathcal{B}_1}}
* {4}
* [*(yellow)]{4}
\quad
$\longrightarrow$
\\\quad\\\quad\\
\ydiagram{3,1+2,2+1}
* [*(yellow)]{3,1+2,2+1}
\quad,\quad
\ytableaushort
{{\mathcal{B}_1}{\mathcal{B}_1}{\mathcal{B}_1}{\mathcal{B}_1}{\mathcal{B}_2}}
* {5}
* [*(green)]{4}* [*(yellow)]{1+4}
\quad
$\longrightarrow$
\ydiagram{3,1+2,2+1}
* [*(yellow)]{3,1+2,2+1}
\quad,\quad
\ytableaushort
{{\mathcal{B}_1}{\mathcal{B}_1}{\mathcal{B}_1}{\mathcal{B}_1}{\mathcal{B}_2},{\mathcal{B}_3}{\mathcal{B}_3}{\mathcal{B}_3}}
* {5,3}
* [*(green)]{5}* [*(yellow)]{0,3}
\quad 
$\longrightarrow$
\\\quad\\\quad\\
\ydiagram{3,1+2,2+1}
* [*(yellow)]{3,1+2,2+1}
\quad,\quad
\ytableaushort
{{\mathcal{B}_1}{\mathcal{B}_1}{\mathcal{B}_1}{\mathcal{B}_1}{\mathcal{B}_2}{\mathcal{B}_4},{\mathcal{B}_3}{\mathcal{B}_3}{\mathcal{B}_3}}
* {6,3}
* [*(green)]{5}* [*(yellow)]{1+5}* [*(green)]{0,3}
\quad 
$\longrightarrow$
\ydiagram{3,1+2,2+1}
* [*(yellow)]{3,1+2,2+1}
\quad,\quad
\ytableaushort
{{\mathcal{B}_1}{\mathcal{B}_1}{\mathcal{B}_1}{\mathcal{B}_1}{\mathcal{B}_2}{\mathcal{B}_4},{\mathcal{B}_3}{\mathcal{B}_3}{\mathcal{B}_3},{\mathcal{B}_5}}
* {6,3,1}
* [*(green)]{6}* [*(green)]{0,3}* [*(yellow)]{0,0,1} 
\\\quad\\\quad\\
$\longrightarrow$
\quad
\ydiagram{3,1+2,2+1}
* [*(yellow)]{3,1+2,2+1}
\quad,\quad
\ytableaushort
{{\mathcal{B}_1}{\mathcal{B}_1}{\mathcal{B}_1}{\mathcal{B}_1}{\mathcal{B}_2}{\mathcal{B}_4},{\mathcal{B}_3}{\mathcal{B}_3}{\mathcal{B}_3},{\mathcal{B}_5}}
* {6,3,1}
* [*(green)]{6,3,1}\\\quad\\\quad\\
\end{example}

\begin{example}\label{eg2}\quad\\\\
Let $\pi_d = (12,11,6,4,2)\in\mathcal{D}_{35,-1}^{6,0}$ so that $l(\pi_d) = 12\le 2N+\nu = 12$. Since $k = -1 <= 0$, by Lemma \ref{relka}, $a = -2k = 2$ which implies $T_2 = 3$ is the triangular part. From the shifted Young diagram of $\pi_d$, $c(\pi_d) = \{1,2,3,4,5,5,4,3,2,2,2,2\}$, and $\Delta = \{3,4,5,5,4,3,2,2,2,2\}$. $\psi_{-}(\pi_d) = (T_2,\pi)$. So, $b_1 = 3$, $b_2 = 1$, $b_3 = 4$, $b_4 = 1$, $b_5 = 3$, $b_6 = 0$, $b_7 = 2$, $b_8 = 0$, and $b_9 = 2$ which implies $\pi = (5,4,3,2,2)$. Clearly, $l(\pi) = 5 < N+\nu-k = 7$ and $\#(\pi) = 5 = N+k$. Hence, $\pi\in\mathcal{P}_{16,7,5}$.\\\par Now, for the reverse direction, we are given $T_2 = 3$ and $\pi = (5,4,3,2,2)\in\mathcal{P}_{16,7,5}$. So, the solutions to $2k^2-k = 3$ are $k = -1$ and $k = \frac{3}{2}$. Since $k\in\mathbb{Z}$, $k = -1\le 0$. On solving $N+\nu-k = 7$ and $N+k = 5$, we have $(N,\nu) = (6,0)$. On solving $\frac{n-3}{2} = 16$, we have $|\pi_d| = n = 35$. Now, $a = -2\cdot (-1) = 2$ since $k = -1\le 0$ which implies $b_1 = a+1 = 3$, $b_2 = 1$, $b_3 = 4$, $b_4 = 1$, $b_5 = 3$, $b_6 = 0$, $b_7 = 2$, $b_8 = 0$, and $b_9 = 2$ following the block diagram configuration in Figure \ref{fig1}. Now, we obtain $d_1 = b_1 = 3$, $d_2 = b_1+b_2 = 4$, $d_3 = b_2+b_3 = 5$, $d_4 = b_3+b_4 = 5$, $d_5 = b_4+b_5 = 4$, $d_6 = b_5+b_6 = 3$, $d_7 = b_6+b_7 = 2$, $d_8 = b_7+b_8 = 2$, $d_9 = b_8+b_9 = 2$, and $d_{10} = b_9+b_{10} = 2$ since $b_{10} = 0$. Thus, we obtain the sequence $\{3,4,5,5,4,3,2,2,2,2\}$ which we write column-wise and if we append columns of length $\{1,2\}$ to the left of the column of length $3$, we retrieve back the shifted Young diagram of the partition $\pi_d = (12,11,6,4,2)\in\mathcal{D}_{35,-1}^{6,0}$.\\\quad\\\quad\\
\ytableausetup{centertableaux}
\ytableaushort
{\none,\none,\none,\none}
* {12,11,6,4,2}
{${}\underset{\psi_{-}^{-1}}{\stackrel{\psi_{-}}{\rightleftarrows}}{}$}
\quad
\ydiagram{2,1+1}
* [*(yellow)]{2,1+1}
\quad,\quad
\ytableaushort
{{\mathcal{B}_1}{\mathcal{B}_1}{\mathcal{B}_1}{\mathcal{B}_2}{\mathcal{B}_4},{\mathcal{B}_3}{\mathcal{B}_3}{\mathcal{B}_3}{\mathcal{B}_3},{\mathcal{B}_5}{\mathcal{B}_5}{\mathcal{B}_5},{\mathcal{B}_7}{\mathcal{B}_7},{\mathcal{B}_9}{\mathcal{B}_9}}
* {5,4,3,2,2}
* [*(green)]{5,4,3,2,2}\\\quad\\\quad\
\end{example}

\begin{example}\label{eg3}\quad\\\\
Let $\pi_d = (11,8,6,5,4,3,2,1)\in\mathcal{D}_{40,-2}^{5,1}$ so that $l(\pi_d) = 11\le 2N+\nu = 11$. Since $k = -2 <= 0$, by Lemma \ref{relka}, $a = -2k = 4$ which implies $T_4 = 10$ is the triangular part. From the shifted Young diagram of $\pi_d$, $c(\pi_d) = \{1,2,3,4,5,6,7,8,2,1,1\}$, and $\Delta = \{5,6,7,8,2,1,1\}$. $\psi_{-}(\pi_d) = (T_4,\pi)$. So, $b_1 = 5$, $b_2 = 1$, $b_3 = 6$, $b_4 = 2$, $b_5 = 0$, and $b_6 = 1$ which implies $\pi = (8,7)$. Clearly, $l(\pi) = 8 = N+\nu-k$ and $\#(\pi) = 2 < N+k = 3$. Hence, $\pi\in\mathcal{P}_{15,8,3}$.\\\par Now, for the reverse direction, we are given $T_4 = 10$ and $\pi = (8,7)\in\mathcal{P}_{15,8,3}$. So, the solutions to $2k^2-k = 10$ are $k = -2$ and $k = \frac{5}{2}$. Since $k\in\mathbb{Z}$, $k = -2\le 0$. On solving $N+\nu-k = 8$ and $N+k = 3$, we have $(N,\nu) = (5,1)$. On solving $\frac{n-10}{2} = 15$, we have $|\pi_d| = n = 40$. Now, $a = -2\cdot (-2) = 4$ since $k = -2\le 0$ which implies $b_1 = a+1 = 5$, $b_2 = 1$, $b_3 = 6$, $b_4 = 2$, $b_5 = 0$, and $b_6 = 1$ following the block diagram configuration in Figure \ref{fig1}. Now, we obtain $d_1 = b_1 = 5$, $d_2 = b_1+b_2 = 6$, $d_3 = b_2+b_3 = 7$, $d_4 = b_3+b_4 = 8$, $d_5 = b_4+b_5 = 2$, $d_6 = b_5+b_6 = 1$, and $d_7 = b_6+b_7 = 1$ since $b_7 = 0$. Thus, we obtain the sequence $\{5,6,7,8,2,1,1\}$ which we write column-wise and if we append columns of length $\{1,2,3,4\}$ to the left of the column of length $5$, we retrieve back the shifted Young diagram of the partition $\pi_d = (11,8,6,5,4,3,2,1)\in\mathcal{D}_{40,-2}^{5,1}$.\\\quad\\\quad\\
\ytableausetup{centertableaux,boxsize=1.3em}
\ytableaushort
{\none,\none,\none,\none}
* {11,8,6,5,4,3,2,1}
{${}\underset{\psi_{-}^{-1}}{\stackrel{\psi_{-}}{\rightleftarrows}}{}$}
\quad
\ydiagram{4,1+3,2+2,3+1}
* [*(yellow)]{4,1+3,2+2,3+1}
\quad,\quad
\ytableaushort
{{\mathcal{B}_1}{\mathcal{B}_1}{\mathcal{B}_1}{\mathcal{B}_1}{\mathcal{B}_1}{\mathcal{B}_2}{\mathcal{B}_4}{\mathcal{B}_6},{\mathcal{B}_3}{\mathcal{B}_3}{\mathcal{B}_3}{\mathcal{B}_3}{\mathcal{B}_3}{\mathcal{B}_3}{\mathcal{B}_4}}
* {8,7}
* [*(green)]{8,7}\\\quad\\\quad\\
\end{example}

\begin{example}\label{eg4}\quad\\\\
Let $\pi_d = (11,8,7,4,3,1)\in\mathcal{D}_{34,2}^{6,1}$ so that $l(\pi_d) = 11\le 2N+\nu = 13$. Since $k = 2 > 0$, by Lemma \ref{relka}, $a = 2k-1 = 3$ which implies $T_3 = 6$ is the triangular part. From the shifted Young diagram of $\pi_d$, $c(\pi_d) = \{1,2,3,4,5,6,5,3,3,1,1\}$, and $\Delta = \{4,5,6,5,3,3,1,1\}$. $\psi_{+}(\pi_d) = (T_3,\pi)$. So, $b_1 = 4$, $b_2 = 1$, $b_3 = 5$, $b_4 = 0$, $b_5 = 3$, $b_6 = 0$, and $b_7 = 1$ which implies $\pi = (5,5,3,1)$. Clearly, $l(\pi) = 5 < N+k = 8$ and $\#(\pi) = 4 < N+\nu-k = 5$. Hence, $\pi\in\mathcal{P}_{14,8,5}$.\\\par Now, for the reverse direction, we are given $T_3 = 6$ and $\pi = (5,5,3,1)\in\mathcal{P}_{14,8,5}$. So, the solutions to $2k^2-k = 6$ are $k = 2$ and $k = -\frac{3}{2}$. Since $k\in\mathbb{Z}$, $k = 2 > 0$. On solving $N+k = 8$ and $N+\nu-k = 5$, we have $(N,\nu) = (6,1)$. On solving $\frac{n-6}{2} = 14$, we have $|\pi_d| = n = 34$. Now, $a = 2\cdot 2-1 = 3$ since $k = 2 > 0$ which implies $b_1 = a+1 = 4$, $b_2 = 1$, $b_3 = 3$, $b_4 = 0$, $b_5 = 3$, $b_6 = 0$, and $b_7 = 1$ following the block diagram configuration in Figure \ref{fig1}. Now, we obtain $d_1 = b_1 = 4$, $d_2 = b_1+b_2 = 5$, $d_3 = b_2+b_3 = 6$, $d_4 = b_3+b_4 = 5$, $d_5 = b_4+b_5 = 3$, $d_6 = b_5+b_6 = 3$, $d_7 = b_6+b_7 = 1$, and $d_8 = b_7+b_8 = 1$ since $b_8 = 0$. Thus, we obtain the sequence $\{4,5,6,5,3,3,1,1\}$ which we write column-wise and if we append columns of length $\{1,2,3\}$ to the left of the column of length $4$, we retrieve back the shifted Young diagram of the partition $\pi_d = (11,8,7,4,3,1)\in\mathcal{D}_{34,2}^{6,1}$.\\\quad\\\quad\\
\ytableausetup{centertableaux}
\ytableaushort
{\none,\none,\none,\none}
* {11,8,7,4,3,1}
{${}\underset{\psi_{+}^{-1}}{\stackrel{\psi_{+}}{\rightleftarrows}}{}$}
\quad
\ydiagram{3,1+2,2+1}
* [*(yellow)]{3,1+2,2+1}
\quad,\quad
\ytableaushort
{{\mathcal{B}_1}{\mathcal{B}_1}{\mathcal{B}_1}{\mathcal{B}_1}{\mathcal{B}_2},{\mathcal{B}_3}{\mathcal{B}_3}{\mathcal{B}_3}{\mathcal{B}_3}{\mathcal{B}_3},{\mathcal{B}_5}{\mathcal{B}_5}{\mathcal{B}_5},{\mathcal{B}_7}}
* {5,5,3,1}
* [*(green)]{5,5,3,1}\\\quad\\
\end{example}

\section{Concluding remarks}\label{s5}
\begin{enumerate}
    \item We get the bounds on the largest part and the number of parts of the image partition $\pi$ from the $q$-binomial coefficient on the right-hand side of (\ref{eq1}). However, it will be interesting to examine the conditions on $\pi_d$ under which the bounds on both the largest part and the number of parts of the image partition $\pi$, as in the statement of Theorem \ref{combbu}, become exact equalities. One may even like to investigate conditions on $\pi_d$ under which any one of the two bounds, i.e., either the bound on the largest part or the bound on the number of parts of $\pi$ become an exact equality.\\ 
    \item It will be worth finding an exact formula (or at least the generating function) of the number of strict partitions of an integer $N$ with fixed BG-rank $k$, fixed largest part $L$, and fixed number of parts $M$.\\
    \item Let $\nu\in\{0,1\}$, $N$ be any non-negative integer and $k$ be any integer. If $\Tilde{B}_N(k,q)$ denotes the generating function for the number of partitions into parts less than or equal to $N$ with BG-rank equal to $k$, then Berkovich and Uncu \cite[Theorem $3.2$]{Ber-Unc16} showed that \begin{align*}
        \Tilde{B}_{2N+\nu}(k,q) = \dfrac{q^{2k^2-k}}{(q^2;q^2)_{N+k}(q^2;q^2)_{N+\nu-k}}.\numberthis\label{eq51}
    \end{align*} Summing over all values of $k$ in (\ref{eq1}), we get \cite[Theorem $3.3$]{Ber-Unc16}\begin{align*}
        \sum\limits_{k=-N}^{N+\nu}q^{2k^2-k}\left[\begin{matrix}2N+\nu\\N+k\end{matrix}\right]_{q^2} = (-q;q)_{2N+\nu}.\numberthis\label{eq52}
    \end{align*} Using (\ref{eq51}) and (\ref{eq52}), one then gets a proof of the following identity \cite[Corollary $3.4$]{Ber-Unc16} \begin{align*}
        \sum\limits_{k=-N}^{N+\nu}\dfrac{q^{2k^2-k}}{(q^2;q^2)_{N+k}(q^2;q^2)_{N+\nu-k}} = \dfrac{1}{(q;q)_{2N+\nu}}.\numberthis\label{eq53}
    \end{align*} It will be interesting to look at a direct combinatorial (or more specifically bijective) proof of (\ref{eq51}) and (\ref{eq53}).\\ 
\end{enumerate}

\section{Acknowledgments}
The authors would like to thank Alexander Berkovich for encouraging them to prove (\ref{eq1}) using combinatorial methods and for his very helpful comments and suggestions. The authors would also like to thank George Andrews for his kind interest and Ali Uncu for previewing a preliminary draft of this paper and for his helpful suggestions. The authors would also like to thank the referee for helpful comments and suggestions.

\bibliographystyle{amsplain}


\end{document}